\documentclass[12pt]{amsart}
\usepackage[utf8]{inputenc}
\usepackage{amsmath,amstext,amsopn,amssymb, amsfonts, amsthm, mathtools} 
\usepackage{xcolor}
\usepackage{geometry}
\usepackage{bbm}
\usepackage{pdfsync}

\usepackage[pdftex,bookmarks,colorlinks,breaklinks]{hyperref}  
\definecolor{dullmagenta}{rgb}{0.4,0,0.4}   
\definecolor{darkblue}{rgb}{0,0,0.4}
\definecolor{darkgreen}{rgb}{0,0.4,0}
\hypersetup{linkcolor=darkblue,citecolor=blue,filecolor=dullmagenta,urlcolor=darkblue} 

\newtheorem{theorem}{Theorem}
\newtheorem*{theorem*}{Theorem}
\newtheorem{lemma}[theorem]{Lemma}
\newtheorem*{lemma*}{Lemma}
\newtheorem{proposition}[theorem]{Proposition}
\newtheorem{corollary}[theorem]{Corollary}

\newtheorem{remark}[theorem]{Remark}
\numberwithin{equation}{section}
\numberwithin{theorem}{section}

\makeindex
\begin{document}

\title{Extremizers and sharp weak-type estimates for positive dyadic shifts}

\author{Guillermo Rey}
\address{Department of Mathematics, Michigan State University, East Lansing MI 48824-1027}
\email{reyguill@math.msu.edu}
\thanks{The first author was partially supported by NSF grant DMS-1056965}

\author{Alexander Reznikov}
\address{Department of Mathematics, Michigan State University, East Lansing MI 48824-1027}
\email{rezniko2@msu.edu}

\begin{abstract}
	We find the exact Bellman function for the weak $L^1$ norm of local positive dyadic shifts. We also describe a sequence of functions, self-similar in nature, which in the limit extremize the local weak-type (1,1) inequality.
\end{abstract}

\maketitle

\section{Introduction}

The purpose of this article is to study the weak-type $(1,1)$ boundedness of the operator
\[
	\mathcal{A} f = \sum_{Q \in \mathcal{D}(I)} \alpha_Q \langle f \rangle_Q \mathbbm{1}_Q.
\]
Here $I$ denotes any finite interval in $\mathbb{R}$, $\langle f \rangle_I = \frac{1}{|I|}\int_I f$, $\mathcal{D}(I)$ denotes the dyadic grid consisting of dyadic subintervals of $I$ and $\{\alpha_J\}_{J \in \mathcal{D}(I)}$ 
is a Carleson sequence adapted to $I$, i.e.: $\alpha_J \geq 0$ for all $J \in \mathcal{D}(I)$ and
\[
	\sup_{J \in \mathcal{D}(I)} \frac{1}{|J|} \sum_{K \in \mathcal{D}(J)} \alpha_K |K| = C <\infty.
\]

These operators have recently appeared in the works of A. K. Lerner \cite{Lerner2010} and \cite{Lerner2012}, where $\alpha_K$ was a binary sequence, although the ideas go back to \cite{Garnett1982}. Hence, we will call them \textit{Lerner operators} in the sequel.
Here we find the exact Bellman function describing the local boundedness of $\mathcal{A}$ from $L^1$ to $L^{1,\infty}$.

It is easy to see that the operator $\mathcal{A}$ is bounded in $L^2$. This, together with a decomposition of Calder\'on-Zygmund type, can be used to prove an estimate of the form
\[
	\Bigl|\Bigl\{x\in I: \, \Bigl| \sum_{J \in \mathcal{D}(I)} \alpha_J \langle f \rangle_J \mathbbm{1}_J(x)\Bigr| > \lambda\Bigr\}\Bigr| \leq \frac{C}{\lambda} \int_{I} |f|.
\]
However, here we precisely describe how the best constant in the above inequality changes with respect to the parameters of the problem. 

The main result of the article is the following theorem:
\begin{theorem}\label{MainTheorem}
	Let $A$, $\lambda$ and $t$ be positive numbers and $I$ an interval in $\mathbb{R}$, then
	\[
		\sup \frac{1}{|I|}\Bigl|\Bigl\{x\in I: \, \sum_{J \in \mathcal{D}(I)} \alpha_J \langle f \rangle_J \mathbbm{1}_J(x) > \lambda\Bigr\}\Bigr| =
			\begin{cases}
				\frac{2At}{A\lambda+t} & \text{if } 0 \leq t \leq A\lambda \leq \lambda, \\
				\vspace{-10pt}\\
				\sqrt{\frac{At}{\lambda}} &\text{if } 0 \leq A \leq \min\Bigl( \frac{t}{\lambda}, \frac{\lambda}{t} \Bigr), \\
				1 &\text{otherwise.}
			\end{cases}
	\]
	Where the supremum is taken over all nonnegative functions $f$ with $\langle f \rangle_I = t$ and all nonnegative sequences $\{\alpha_J\}_{J \in \mathcal{D}(I)}$ with Carleson constant at most $1$ which satisfy
	\[
		\frac{1}{|I|} \sum_{J \in \mathcal{D}(I)} \alpha_J|J| = A.
	\]
\end{theorem}
We also provide a sequence of examples which, in the limit, attain the supremum of the previous result. See the last section for details on the structure of such examples.

As an immediate corollary we have the following local weak-type (1,1) estimate:
\begin{corollary}
	For any nonnegative $f \in L^1([0,1))$ and for any Carleson sequence $\{\alpha_J\}_{J \in \mathcal{D}([0,1))}$ with constant at most $1$ we have the sharp bound
	\[
		\bigl| \bigl\{ x\in [0,1):\, \mathcal{A}f(x) > \lambda \bigr\} \bigr| \leq
			\begin{cases}
				\frac{2\|f\|_{L^1}}{\lambda + \|f\|_{L^1}} & \text{if } \|f\|_{L^1} \leq \lambda \\
				1 & \text{if } \|f\|_{L^1} \geq \lambda,
			\end{cases}
	\]
	which in particular implies that
	\[
		\|\mathcal{A}f\|_{L^{1,\infty}([0,1))} \leq 2\|f\|_{L^1([0,1))},
	\]
	and that the constant $2$ is sharp.

\end{corollary}

Operators similar to these were recently studied in \cite{Melas2005}, \cite{Melas2008}, \cite{Melas2009} and \cite{Melas2009a}, however their results are slightly different from ours. They
consider the supremum taken over all functions $f$ satisfying
\[
	\int_I f = s \quad \text{and} \quad \int_I G(f) = t,
\]
where $G$ is a strictly convex function satisfying $G(x)/x \to \infty$ as $x \to \infty$. This does not include the question of boundedness from $L^1$ to $L^{1,\infty}$. Our method of proof is different
than the one used in the articles cited above, where they use the deep combinatorial properties of these operators. See also the monograph \cite{OsekowskiBook} by 
A. Os{\setbox0=\hbox{e}{\ooalign{\hidewidth
    \lower1.5ex\hbox{`}\hidewidth\crcr\unhbox0}}}kowski for related results. We instead follow the ideas in \cite{Slavin2008} and \cite{Vasyunin2009} to solve the Bellman
PDE and prove its sharpness.

This problem is also closely related to studying Haar shifts, the main difference being that Haar shifts are not positive operators. It has been shown however, see \cite{Cruz-Uribe2012}, that
Lerner-type operators can be used to bound Haar shifts. The reader can find results similar to ours in \cite{RVV-UnweightedMartingale}, \cite{NRVV-WeightedHaarShifts} and \cite{NV-Square}.

The article is organized as follows. In Section 2 we explain how the Bellman function technique is used to compute the supremum in Theorem \ref{MainTheorem}. In Section 3 we give a supersolution to the Bellman
variational problem which serves as an upper bound for the exact Bellman function. Finally, in Section 4 we show that the function we found in the previous section is the exact Bellman function, we also 
give a sequence of examples which, in the limit, extremize the inequality of Theorem \ref{MainTheorem}.

\section*{Acknowledgements}
	The authors would like to thank Alexander Volberg for originally proposing the problem and for many valuable discussions.

\section{The Bellman function technique}

Consider the function defined in $\Omega = \{(t, A, \lambda):\, 0 \leq t, 0 \leq A \leq 1, \, \lambda \in \mathbb{R}\}$
\[
	\mathbb{B}(t, A,\lambda) = \sup\Bigl\{ \frac{1}{|I|}\Bigl| \bigl\{x\in I:\, \sum_{J \in \mathcal{D}(I)}\alpha_J \langle f \rangle_J \mathbbm{1}_J > \lambda \bigr\} \Bigr|\Bigr\},
\]
where the supremum is taken over all all nonnegative functions $f$ on $I$ with $\langle f \rangle_I = t$ and all Carleson sequences $\{\alpha_J\}_{j \in \mathcal{D}(I)}$ with constant at most $1$ and
\[
	A = \frac{1}{|I|}\sum_{J \subseteq I} \alpha_J |J|.
\]
Note that $I$ is \emph{not} a parameter in $\mathbb{B}$, this is because the supremum is invariant under dilations and translations in $I$, and hence independent of $I$.

The Bellman function technique, which first appeared in the 1995 preprint version of \cite{Nazarov1999}, is based on showing that $\mathbb{B}$ solves a certain minimization problem. One first shows that 
$\mathbb{B}$ satisfies a kind of concavity property and explicitly computes $\mathbb{B}$ in a subdomain 
natural to the problem (this is usually easy). Then one shows that any continuous positive function satisfying these conditions majorizes $\mathbb{B}$, which reduces the problem to finding the smallest function which satisfies these properties.
Finally one has to actually find such a function, this is usually the hardest part. The reader can find insightful introductions in \cite{Nazarov2001} and \cite{Nazarov1996}, see also \cite{Nazarov1999}, \cite{Slavin2008}, and \cite{Vasyunin2009}
for more examples of this technique.

Let us begin by describing more precisely the concavity property which $\mathbb{B}$ satisfies:
\begin{lemma}[Main inequality]\label{Bellman.MainInequality}
	\begin{equation}\label{Bellman.MainInequality.eq}
		\mathbb{B}(t,A,\lambda) \geq \frac{1}{2}\Bigl( \mathbb{B}(t_1,A_1,\lambda') + \mathbb{B}(t_2,A_2,\lambda') \Bigr)
	\end{equation}
	whenever
	\[
		t = \frac{t_1+t_2}{2}, \quad A = \frac{A_1+A_2}{2} + \alpha \quad \text{and} \quad \lambda = \lambda' + \alpha t
	\]
	and $\alpha \geq 0$.
\end{lemma}
\begin{proof}
	Consider any dyadic interval $I$, any function $f \geq 0$ satisfying
	\[
		\langle f \rangle_{I_-} = t_1 \quad \text{and} \quad \langle f \rangle_{I_+} = t_2
	\]
	and any Carleson sequence $\{\alpha_J\}_{J \in \mathcal{D}(I)}$ with constant at most $1$ on $I$ satisfying
	\[
		\frac{1}{|I_-|} \sum_{J \in \mathcal{D}(I_-)} \alpha_J |J| = A_1, \quad \frac{1}{|I_-|} \sum_{J \in \mathcal{D}(I_+)} \alpha_J |J| = A_2 \quad \text{and} \quad \alpha_I = \alpha.
	\]	
	Suppose also that $\lambda = \lambda' + \alpha t$.
	
	Since $\langle f \rangle_I = t$ then we must have
	\begin{align*}
		\mathbb{B}(t,A,\lambda) &\geq \frac{1}{|I|}\Bigl| \Bigl\{ x\in I: \, \sum_{J \in \mathcal{D}(I)} \alpha_J \langle f \rangle_J \mathbbm{1}_J(x) > \lambda \Bigr\} \Bigr|
	\end{align*}
	since the supremum defining $\mathbb{B}$ is taken over a larger space.
	
	Observe now that
	\begin{multline*}
		\frac{1}{|I|}\Bigl| \Bigl\{ x\in I: \, \sum_{J \in \mathcal{D}(I)} \alpha_J \langle f \rangle_J \mathbbm{1}_J(x) > \lambda \Bigr\} \Bigr| = \\
			\frac{1}{2|I_-|}\Bigl| \Bigl\{ x\in I_- : \, \sum_{J \in \mathcal{D}(I)} \alpha_J \langle f \rangle_J \mathbbm{1}_J(x) > \lambda \Bigr\} \Bigr| + \\
			\frac{1}{2|I_+|}\Bigl| \Bigl\{ x\in I_+ : \, \sum_{J \in \mathcal{D}(I)} \alpha_J \langle f \rangle_J \mathbbm{1}_J(x) > \lambda \Bigr\} \Bigr| \\
			=\frac{1}{2|I_-|}\Bigl| \Bigl\{ x\in I_- : \, \sum_{J \in \mathcal{D}(I_-)} \alpha_J \langle f \rangle_J \mathbbm{1}_J(x) > \lambda - \alpha_I t \Bigr\} \Bigr| + \\
			\frac{1}{2|I_+|}\Bigl| \Bigl\{ x\in I_+ : \, \sum_{J \in \mathcal{D}(I_+)} \alpha_J \langle f \rangle_J \mathbbm{1}_J(x) > \lambda - \alpha_I t \Bigr\} \Bigr| \\
			=\frac{1}{2|I_-|}\Bigl| \Bigl\{ x\in I_- : \, \sum_{J \in \mathcal{D}(I_-)} \alpha_J \langle f \rangle_J \mathbbm{1}_J(x) > \lambda' \Bigr\} \Bigr| +\\ 
			\frac{1}{2|I_+|}\Bigl| \Bigl\{ x\in I_+ : \, \sum_{J \in \mathcal{D}(I_+)} \alpha_J \langle f \rangle_J \mathbbm{1}_J(x) > \lambda' \Bigr\} \Bigr|
	\end{multline*}
	and thus the claim follows.
\end{proof}

Also, we trivially see that $\mathbb{B}$ must satisfy the following ``obstacle'' condition:
\begin{equation}\label{Bellman.Obstacle}
	\mathbb{B}(t,A,\lambda) = 1 \quad \text{whenever } \lambda < 0.
\end{equation}

As we described in the beginning of the section, the function $\mathbb{B}$ is a minimizer in the space of positive functions which satisfy these properties. The following proposition makes this precise:
\begin{proposition} \label{Bellman.BellmanDoesIt}
	Suppose a continuous function $F$ satisfies inequality \eqref{Bellman.MainInequality.eq} together with the obstacle condition \eqref{Bellman.Obstacle}, then we must have
	\[
		\mathbb{B}(t,A,\lambda) \leq F(t,A,\lambda).
	\]
\end{proposition}
\begin{proof}
	Let $f \geq 0$ be an integrable function on an interval $I$ and let $\{\alpha_J\}_{J \in \mathcal{D}(I)}$ be a Carleson sequence with constant at most $1$, then for all fixed $\lambda$ we have (by \eqref{Bellman.MainInequality.eq})
	\begin{align*}
		F(\langle f \rangle_I, A, \lambda) &= F\biggl( \frac{\langle f \rangle_{I_-} + \langle f \rangle_{I_+}}{2}, \frac{A_-+A_+}{2} + \alpha_I, \lambda \biggr) \\
		&\geq \frac{1}{2}\Bigl( F(\langle f \rangle_{I_-}, A_-, \lambda - \alpha_I \langle f \rangle_I) + F(\langle f \rangle_{I_+}, A_+, \lambda - \alpha_I \langle f \rangle_I) \Bigr),
	\end{align*}
	where $A = \frac{1}{|I|}\sum_{J \subseteq I} \alpha_J |J|$ and $A_{\pm}$ is defined analogously for $I_-$ and $I_+$.
	
	If we iterate this inequality we obtain
	\begin{align*}
		F(\langle f \rangle_I, A, \lambda) &\geq \frac{1}{2^N} \sum_{J \subset I,\, |J| = 2^{-N}|I|} F(\langle f \rangle_J, A_J, \lambda - \sum_{k=1}^N \alpha_{J^{(k)}} \langle f \rangle_{J^(k)} \mathbbm{1}_{J^{(k)}}(c_J) ),
	\end{align*}
	where $A_J = \frac{1}{|J|} \sum_{P \subseteq J} \alpha_P |P|$.
	
	If we assume a priori that the Carleson sequence $\alpha$ is finite then we can let $N \to \infty$ and obtain
	\begin{align*}
		F(\langle f \rangle_I, A, \lambda) &\geq \frac{1}{|I|} \int_I F(f(x), A(x), \lambda - \mathcal{A}f(x) ) \, dx \\
		&\geq \frac{1}{|I|} \int_{\{x \in I : \lambda - \mathcal{A}f(x) < 0\}} 1 \, dx &&\text{by } \eqref{Bellman.Obstacle}\\
		&= \frac{1}{|I|}|\{x\in I: \mathcal{A}f(x) > \lambda \}|.
	\end{align*}
	Here $A(x)$ is almost everywhere-defined as the limit of $A(J)$ as $J \to x$, this is easily seen to exist almost everywhere by the Lebesgue differentiation theorem.

	Letting the number of non-zero elements of $\{\alpha_J \}_{J \in \mathcal{D}(I)}$ tend to infinity and then taking the supremum in the definition of $\mathbb{B}$ we obtain
	\[
		F(\langle f \rangle_I, A, \lambda) \geq \mathbb{B}(\langle f \rangle_I, A, \lambda).
	\]
\end{proof}
\begin{remark}
	Note that we don't know yet if the function $\mathbb{B}$ is continuous, thus finding a minimizer in the space of continuous functions might not give us the true Bellman function. It turns out, however, that assuming continuity
	(actually $C^1$ smoothness) we are able to find a positive function satisfying \eqref{Bellman.MainInequality.eq} and \eqref{Bellman.Obstacle} which moreover is best possible without the a priori assumption of smoothness.
	We show this in the last section.
\end{remark}

We have therefore seen that finding any positive continuous function $F$ satisfying \eqref{Bellman.MainInequality.eq} and \eqref{Bellman.Obstacle} will give us an upper bound for $\mathbb{B}$. In the next section we find such a function.
\section{Finding the Bellman function candidate}

Our goal now is to find the smallest continuous function $F$ satisfying \eqref{Bellman.MainInequality.eq} and \eqref{Bellman.Obstacle}. As we remarked after Proposition \ref{Bellman.BellmanDoesIt}, we will assume a priori
that $F$ is $C^1$. Moreover, we will restrict the minimization space even more by requiring $F$ to have the same kind of homogeneity that the true $\mathbb{B}$ must have, i.e.:
\[
	\mathbb{B}(\eta t, A, \eta \lambda) = \mathbb{B}(t,A,\lambda) \quad \forall \eta >0, \lambda >0.
\]
This in principle might make our candidate for Bellman function larger than the one we could find without requiring such homogeneity. However, the optimal Bellman function satisfies this identity, so requiring $F$ to also satisfy it will not
prevent us from finding it.

Assuming smoothness we can write the Main Inequality \eqref{Bellman.MainInequality.eq} as a concavity condition, together with a monotonicity property along certain characteristics. More precisely, if $F$ is a smooth positive function, then
\eqref{Bellman.MainInequality.eq} together with \eqref{Bellman.Obstacle} and the above homogeneity is equivalent to the following conditions:
\begin{enumerate}
	\item $F$ is nonnegative, and concave in the first two variables.
	\item $F(t,A,\lambda)$ is increasing in the direction $(0,1,t)$.
	\item $F(st,A,s\lambda) = F(t,A,\lambda)$ for all $s > 0$.
	\item $F(t,A,\lambda) = 1$ whenever $\lambda <0$
\end{enumerate}
Indeed, if we let $\alpha = 0$ in \eqref{Bellman.MainInequality.eq} we see that $\mathbb{B}$ is concave in the variables $(t,A)$. If we set $A_1 = A_2 = A$ and $t_1 = t_2 = t$ 
then we see, by varying $\alpha$, that $\mathbb{B}(t,A,\lambda)$ is increasing in the direction $(0,1,t)$. This shows that any smooth $F$ satisfying \eqref{Bellman.MainInequality.eq} and \eqref{Bellman.Obstacle}, and which is also homogeneous in
the above sense, must also satisfy properties (1) through (4). Moreover, if $F$ is any smooth function satisfying properties (1) through (4), then it also must satisfy the main inequality \eqref{Bellman.MainInequality.eq} and the obstacle condition
\eqref{Bellman.Obstacle}. To see this observe that using property (1) we obtain \eqref{Bellman.MainInequality.eq} but with $\alpha = 0$, now property (2) allows us to insert an $\alpha$ as in the hypotheses for the main inequality since it describes
the path along which $F$ is increasing. The homogeneity and obstacle conditions are exactly (3) and (4) respectively, so this proves the equivalence.

Using the homogeneity property, we can reduce to finding $M:(0,\infty) \times [0,1] \to [0,\infty)$
such that if
\[
	F(x,y,z) = \begin{cases}
					M(x/z,y) & \text{if } z > 0\\
					1 & \text{if } z \leq 0,
	           \end{cases}
\]
then $F$ satisfies (1) through (4). These properties, when translated to the function $M$, become:
\begin{enumerate}
	\item $M$ is concave.
	\item $M_y - x^2 M_x \geq 0$.
	\item $M(x,y) \to 1$ when $x \to \infty$.
\end{enumerate}
The second of these properties tells us that $M$ is increasing along the characteristics
\[
	\begin{cases}
		\dot{x}(t) &= -x^2 \\
		\dot{y}(t) &= 1.
	\end{cases}
\]
Observe that these characteristics foliate $[0,\infty) \times[0,1]$. Also, if we move backwards in time along a characteristic which starts at $(x_0,1)$ with $x_0 \geq 1$, then this characteristic is above the curve $y = \frac{1}{x}$
and furthermore the characteristic tends to $(\infty, y_f)$ for some $0< y_f<1$. Using the fact that $M(x,y) \to 1$ as $x \to \infty$ and that we should decrease if we move backwards along these characteristics, we must have
\[
	M(x,y) \geq 1 \quad \text{whenever } y \geq \frac{1}{x}.
\]
However, we may assume (if our goal is to find the true Bellman function) that $M \leq 1$ since the true Bellman function $\mathbb{B}$ obviously cannot be larger than $1$, so we will actually impose
\[
	M(x,y) = 1 \quad \text{whenever } y \geq \frac{1}{x}.
\]

Observe that $\mathbb{B}(0,0,1)$ is $0$ and consider the straight line joining the point $(0,0)$ with $(x_1,y_1)$, where $x_1y_1 = 1$. Observe also that the pointwise minimum of any two positive continuous functions
satisfying \eqref{Bellman.MainInequality.eq} and \eqref{Bellman.Obstacle} will give us a smaller function which also satisfies these properties.

We know that the function $M$ should be $1$ at $(x_1,y_1)$ and that, along this line, $M$ should be concave.
The smallest concave curve joining these two points is obviously a straight line, so if defining $M$ in this way produces a smooth concave function satisfying the monotonicity property (2) then the optimal $M$ should be such a function.
Joining the point $(0,0)$ with the points $(x_1,y_1) \in [0,\infty) \times [0,1]$ satisfying $x_1y_1 = 1$ covers everything in the subdomain $0 \leq y \leq \min(x,x^{-1})$, so let us define $M$ here by
\[
	M(x,y) = \sqrt{xy}.
\]
This function is linear along straight lines joining $(0,0)$ with the boundary curve $xy = 1$ and is $1$ at this boundary. It is furthermore concave and satisfies the monotonicity property (2), so if we knew that $\mathbb{B}$ is
continuous then $\mathbb{B}(x,y,1)$ must be defined as above in this subdomain.

We are therefore left with defining $M$ in the upper triangle $\Omega_T = \{0 \leq x \leq y \leq 1\}$. Inspired by the linear behavior of $M$ in the first domain, we make the ansatz that $M$ is actually $1$-homogeneous in the whole domain. 

Let $f(x) = M(x,1)$ for $0 \leq x \leq 1$, then if $M$ is $1$-homogeneous we should have
\[
	M(x,y) = y f(x/y).
\]
If we want condition $(2)$ to hold then we should have
\[
	f(x/y)-(x/y)f'(x/y) - x^2f'(x/y) \geq 0.
\]
We expect this to be an equality on the boundary, which is when $y = 1$, so we will assume that
\[
	f(x) - f'(x)(x+x^2) = 0.
\]
This ordinary differential equation has the solutions
\[
	f(x) = C\frac{x}{1+x},
\]
and we should furthermore have $f(1) = M(1,1) = 1$. So $ C= 2$ and therefore
\[
	f(x) = \frac{2x}{x+1} \implies M(x,y) = \frac{2xy}{x+y}
\]
whenever $1 \geq y \geq x \geq 0$. One easily verifies that $M$ satisfies all the requirements in this subdomain, so we just have to show that the whole function $M$ is concave, but this immediately follows from the fact that $M$ is concave in each
subdomain and that $M$ is $C^1$ (as can be easily seen).

This gives us that
\[
	M(x,y) = \begin{cases}
	         	\frac{2xy}{x+y} &\text{if } 0 \leq x \leq y \leq 1 \\
	         	\sqrt{xy} & \text{if } 0 \leq y \leq \min(x,x^{-1}),
	         \end{cases}
\]
which using the homogeneity gives us the full function of Theorem \ref{MainTheorem}.

\section{Optimality}

In this section we show that the function found in the previous section is actually the exact Bellman function. We first we need a simple technical lemma which will allow us to deduce that $\mathbb{B}(\cdot,\cdot,1)$ must be superlinear
along lines joining $(0,0,1)$ to $(x,1,1)$.

\begin{lemma}\label{DydadicConcavity.Lemma}
	Let $f:[0,1] \to [0,\infty)$ be a function which satisfies
	\begin{equation} \label{DyadicConcavity.eq}
		f\Bigl( \frac{x+y}{2} \Bigr) \geq \frac{1}{2}f(x) + \frac{1}{2}f(y)
	\end{equation}
	for all $0 \leq x \leq y \leq 1$. Then we must have
	\[
		f(x) \geq f(1)x
	\]
	for all $x \in [0,1].$
\end{lemma}
\begin{proof}
	We can assume without loss of generality that $x \in (0,1)$ and that $f(1) = 1$. Using \eqref{DyadicConcavity.eq} we have
	\begin{equation}\label{DyadicConcavity.eq2}
		f(x_0 + \lambda(1-x_0)) \geq \lambda
	\end{equation}
	for all dyadic rationals $\lambda \in [0,1]$, i.e.: numbers of the form $\lambda = k2^{-N}$ for $0 \leq k \leq 2^N$.
	
	For every $N \in \mathbb{N}$ let $k_N$ be the unique integer in $0 \leq k \leq 2^N x$ which satisfies
	\[
		\Bigl| x - \frac{k}{2^N} \Bigr| < \frac{1}{2^N}
	\]
	(this exists because the sequence $k \mapsto k2^{-N}$ is an arithmetic sequence of step $2^{-N}$).
	
	Observe that then, if we define
	\[
		x_N := \frac{2^Nx - k_N}{2^N - k_N} = \frac{x-\frac{k_N}{2^N}}{1-\frac{k_N}{2^N}},
	\]
	we must have $0 \leq x_N \leq \frac{1}{2^N(1-x)}$, so in particular $x_N \to 0$ as $N \to \infty$.
	
	But then
	\[
		\lambda := \frac{k_N}{2^N} = \frac{x-x_N}{1-x_N}
	\]
	is a dyadic rational and plugging it into \eqref{DyadicConcavity.eq2}, with $x_N$ playing the role of $x_0$, yields
	\[
		f(x) \geq \frac{x-x_N}{1-x_N},
	\]
	so letting $N \to \infty$ completes the proof.
\end{proof}

Using this lemma, together with the Main Inequality \eqref{Bellman.MainInequality.eq} we immediately have the following corollary:
\begin{corollary} \label{EasyObstacle}
	We have the following identity:
	\[
		\mathbb{B}(x,y,1) = M(x,y)
	\]
	for all $x,y$ in the subdomain $0 \leq y \leq \min(x,x^{-1})$.
\end{corollary}
\begin{proof}
	We showed in the previous section that $\mathbb{B}(x,y,1) \leq M(x,y)$ for all $(x,y) \in \Omega'$. To show the reverse inequality notice that the Main Inequality \eqref{Bellman.MainInequality.eq} together with
	Lemma \ref{DydadicConcavity.Lemma} imply
	\begin{equation}\label{EasyObstacle.eq}
		\mathbb{B}(x,y,1) \geq \lambda \mathbb{B}\Bigl(\frac{x}{\lambda},\frac{y}{\lambda},1 \Bigr).
	\end{equation}
	We would be done if we can show that $\mathbb{B}(x,y,1) = 1$ whenever $xy = 1$. Indeed, then we can just use equation \eqref{EasyObstacle.eq} with $\lambda = \sqrt{xy}$.
	
	Fix $(x,y) \in \Omega'$ with $xy = 1$ and consider the function
	\[
		f_n = \frac{2^nx}{2^n -1}\mathbbm{1}_{[0,1-2^{-n})}.
	\]
	If $I$ is the interval $[0,1)$ then obviously $\langle f_n \rangle_I = x$. Consider also the Carleson sequence $\{\alpha_J\}_{J \in \mathcal{D}(I)}$ defined by
	\[
		\alpha_J = \begin{cases}
						\frac{y}{1-2^{-n}} & \text{if } J = [2^{-n}(k-1),2^{-n}k) \text{ and } k \in \{1, \dots, 2^n-1\}\\
						0 & \text{otherwise.}
		           \end{cases}
	\]
	Then we have
	\[
		\frac{1}{|I|}\sum_{J \in \mathcal{D}(I)} \alpha_J |J| = y\sum_{k=1}^{2^n-1} \frac{2^{-n}}{1-2^{-n}} = y.
	\]
	
	Also,
	\[
		\mathcal{A}f_n(t) = \begin{cases}
								\Bigl(\frac{2^n}{2^n-1}\Bigr)^2 & \text{if } 0 \leq t < 1- 2^{-n} \\
								0 & \text{otherwise,}
		                    \end{cases}
	\]
	hence
	\[
		\mathbb{B}(x,y,1) \geq 1-2^{-n}
	\]
	for all $n \geq 1$. Letting $n \to \infty$ yields the claim.
\end{proof}
\begin{remark}
	Observe that using the constant function $f(t) = x\mathbbm{1}_I(t)$ and the one-term Carleson sequence which is $y$ on $I$ and $0$ everywhere else, one obtains that $\mathcal{A}f = xy\mathbbm{1}_I$, hence $\mathbb{B}(x,y,1) = 1$ for all $xy >1$.
\end{remark}

Using Lemma \ref{DydadicConcavity.Lemma} in the same way, we just have to show that $\mathbb{B}(x,1,1) = \frac{2x}{x+1}$ to prove that $\mathbb{B}(x,y,1) = M(x,y)$ in the rest of the domain, however this turns out to be harder.

\begin{theorem}\label{MainTheorem2}
	Fix $x \in (0,1)$ and let $\epsilon > 0$. For any interval $I$ there exists a nonnegative function $f$ on $I$ with $\langle f \rangle_I = x$ and a Carleson sequence $\{\alpha_J\}_{J \in \mathcal{D}(I)}$ with Carleson
	constant at most one and verifying
	\[
		\frac{1}{|I|}\sum_{J \in \mathcal{D}(I)} \alpha_J |J| = 1
	\]
	such that
	\[
		\frac{1}{|I|}\Bigl| I \cap \Bigl\{ \sum_{J \in \mathbb{D}(I)} \alpha_J \langle f\rangle_J \mathbbm{1}_J > 1 \Bigr\} \Bigr| = \frac{2x}{x+1} + O(\epsilon).
	\]
\end{theorem}

To prove this we will use the Main Inequality \eqref{Bellman.MainInequality.eq} iteratively to give a decomposition of $f$ consisting of constant functions on certain dyadic intervals, this also gives us the construction of
the sequence $\{\alpha_J\}_{J \in \mathcal{D}(I)}$. The basic idea is to, starting with a point $(x,1)$ in $\Omega'$, use \eqref{Bellman.MainInequality.eq} to split this point into another point $(x_+,1)$ on the boundary and
some point $(x_-,A_-)$. The point $(x_-,A_-)$ is then absorbed back into the initial point and we apply the same procedure to the point $(x_+,1)$ until we get to a point past the obstacle $xy \geq 1$ (where extremizers consist of constant
functions together with one-term Carleson sequences as in the Remark after Corollary \ref{EasyObstacle}).

In order to illustrate the idea we will first prove the lower bound for $\mathbb{B}$ without explicitly constructing the example. The way in which we prove the lower bound will make the construction more intuitive.

\begin{theorem}\label{Optimality}
	The Bellman function $\mathbb{B}$ satisfies
	\[
		\mathbb{B}(x,1,1) = \frac{2x}{x+1}
	\]
	for all $x \in [0,1]$.
\end{theorem}
\begin{proof}
	Let $E(x,y) = \mathbb{B}(x,y,1)$, then using the Main Inequality \eqref{Bellman.MainInequality.eq} we see that we have the following behavior:
	\[
		E(t,1) \geq \frac{1}{2}E\Bigl( \frac{t_1}{1-\alpha t}, A_1 \Bigr) + \frac{1}{2} E\Bigl(\frac{t_2}{1-\alpha t}, A_2 \Bigr) 
	\]
	whenever $t = \frac{t_1+t_2}{2}$ and $1 = \frac{A_1 + A_2}{2}+\alpha$. Letting $\epsilon > 0$, $x = t$ and $A_2 = 1$ we get
	\[
		E(x,1) \geq \frac{1}{2}\Bigl( E\Bigl( x - 2\epsilon, 1 - \frac{2\epsilon}{x} \Bigr) + E\Bigl( x_+, 1 \Bigr)\Bigr),
	\]
	where
	\[
		x_+ = x \frac{1+\epsilon}{1-\epsilon} + 2\epsilon.
	\]

	Since $\mathbb{B}$ is superlinear in the first two variables and $\mathbb{B}(0,0,1) = 0$, we must have
	\[
		E\Bigl( x - 2\epsilon, 1 - \frac{2\epsilon}{x} \Bigr) \geq \Bigl( 1-\frac{2\epsilon}{x} \Bigr)E(x,1)
	\]
	so putting everything together we obtain
	\begin{equation}\label{Bellman.Optimality.It2}
		E(x,1) \geq \frac{x}{x+2\epsilon} E(x_+,1).
	\end{equation}

	If we define inductively $x_{n+1} = x_n \frac{1+\epsilon}{1-\epsilon} + 2\epsilon$ and $x_0 = x$, then we easily see that
	\[
		x_n = \delta^n \Bigl( \frac{1}{1-\epsilon} + x \Bigr) - \frac{1}{1-\epsilon},
	\]
	where $\delta = \frac{1+\epsilon}{1-\epsilon}$.

	We want to stop the iteration once $x_n \geq 1$, and this happens when
	\[
		\delta^n \geq \frac{2-\epsilon}{1+x(1-\epsilon)},
	\]
	let $N = N(\epsilon,x)$ be the smallest integer for which the above inequality does not hold. Then iterating \eqref{Bellman.Optimality.It2} $N$ times we get (since $E(1,1) = 1$)
	\[
		E(x,1) \geq \prod_{j=0}^N \frac{x_j}{x_j+2\epsilon},
	\]
	it just suffices to give a lower bound for the right hand side.

	To this end observe that
	\begin{align*}
		\prod_{j=0}^N \frac{x_j}{x_j+2\epsilon} &\geq \exp\Bigl(-\sum_{j=0}^N \log\Bigl(1 + \frac{2\epsilon}{x_j}\Bigr)\Bigr) \\
		&\geq \exp \Bigl( - \sum_{j=0}^N \frac{2\epsilon}{x_j} \Bigr) \\
		&= \exp\Bigl( -2\epsilon \sum_{j=0}^N \frac{1}{x_j} \Bigr).
	\end{align*}

	Let us estimate $-2\epsilon \sum_{j=0}^N \frac{1}{x_j}$. Using the explicit formula for $x_n$ we have
	\begin{align*}
		-2\epsilon \sum_{j=0}^N \frac{1}{x_j} &= -2\epsilon \sum_{j=0}^N \frac{1}{\delta^j \bigl( \frac{1}{1-\epsilon} + x \bigr)- \frac{1}{1-\epsilon}} \\
		&= -2\epsilon\sum_{j=0}^N \biggl( \frac{1}{\delta^j \bigl( \frac{1}{1-\epsilon} + x \bigr)- \frac{1}{1-\epsilon}} - \frac{1}{\delta^j \bigl( 1 + x \bigr)- 1}\biggr) 
			+ \sum_{j=0}^N \frac{2\epsilon}{1-\delta^j(1+x)}.
	\end{align*}
	The first term tends to $0$ as $\epsilon \to 0$ and the second is a Riemann sum, indeed (recalling the definition of $N = N(x,\epsilon)$:
	\begin{align*}
		\sum_{j=0}^N \frac{2\epsilon}{1-\delta^j(1+x)} &= (1-\epsilon)\sum_{j=0}^N \frac{\delta^j\frac{2\epsilon}{1-\epsilon}}{\delta^j(1- \delta^j (1+x))} \\
		&= (1-\epsilon) \sum_{j=0}^N f(\delta^j)(\delta^{j+1}-\delta^j) \\
		&= \int_1^{\frac{2}{1+x}} f(y) \, dy + O(\epsilon),
	\end{align*}
	as $\epsilon \to 0$ and where
	\[
		f(y) = \frac{1}{y(1-y(x+1))}.
	\]
	It is easy to see that
	\[
		\int_1^{\frac{2}{1+x}} \frac{1}{y(1-y(x+1))} \, dy = \log\Bigl( \frac{2x}{x+1} \Bigr),
	\]
	which completes the proof of the lower bound.

\end{proof}

Let us now use these ideas to construct the example. There are two basic steps in the iteration: first we split the point $(x,1)$ into $(x_-,A_-)$ and $(x_+,1)$, then we absorb $(x_-,A_-)$ into $(x,1)$ and obtain a lower bound for $E(x,1)$
in terms of $E(x_+,1)$, we then iterate this until $x_+ >1$, where we stop because we know that $E(x_+,1)$ must be $1$ there. These two steps are imposing a certain self-similarity on $f$ and the Carleson sequence $\alpha$ in terms
of $(f_+,\alpha_+)$. The following Lemma, which is based on the ideas from \cite{Vasyunin2009}, makes this precise.

\begin{lemma} \label{TechnicalLemma}
	Fix an interval $I$ and let $g_+$ be a nonnegative function on $I_+$. Suppose also that $\alpha^+$ is a Carleson sequence adapted to $I_+$ with constant at most $1$ and such that
	\[
		\frac{1}{|I_+|} \sum_{J \in \mathcal{D}(I_+)} \alpha^+_J |J| = 1.
	\]
	If $\langle g_+ \rangle_{I_+} = x \frac{1+\epsilon}{1-\epsilon} + 2\epsilon$ for some $x \in (0,1)$ and a sufficiently small $\epsilon > 0$, then we can construct a function $f$ on $I$ and a Carleson sequence $\alpha$ adapted to $I$ with constant at most $1$ such that
	$\langle f \rangle_I = x$,
	\begin{equation} \label{TechnicalLemma.eq.1}
		\frac{1}{|I|} \sum_{J \in \mathcal{D}(I)} \alpha_J |J| = 1
	\end{equation}
	and
	\begin{equation} \label{TechnicalLemma.eq.2}
		\frac{1}{|I|} \Bigl| I \cap \Bigl\{ \sum_{J \in \mathcal{D}(I)} \alpha_J \langle f \rangle_J \mathbbm{1}_J > 1  \Bigr\} \Bigr| \geq \Bigl(\frac{x}{x+2\epsilon}\Bigr)
		\frac{1}{|I_+|} \Bigl| I_+ \cap \Bigl\{ \sum_{J \in \mathcal{D}(I_+)} \alpha_J^+ \langle g_+ \rangle_J \mathbbm{1}_J > 1  \Bigr\} \Bigr|.
	\end{equation}
\end{lemma}
\begin{proof}
	We will assume without loss of generality that $I = [0,1)$, also denote $\alpha = \frac{\epsilon}{x}$. Define $\alpha_J$ to be $\alpha$ if $J = I$ and $\alpha^+_J$ for $J \in \mathcal{D}(I_+)$.
	
	Define $f$ to be $(1-\epsilon)g_+$ on $I_+$ and denote $g_- = (1-\epsilon)^{-1}f \mathbbm{1}_{I_-}$, then
	\begin{multline*}
		\frac{1}{|I|}\Bigl| y \in I:\, \sum_{J \in \mathcal{D}(I)} \alpha_J \langle f \rangle_J \mathbbm{1}_J(y) > 1 \Bigr| = \\
		\frac{1}{2|I_-|}\Bigl| y \in I_-:\, \sum_{J \in \mathcal{D}(I_-)} \alpha_J \langle f \rangle_J \mathbbm{1}_J(y) > 1 - \epsilon \Bigr|\\
		+\frac{1}{2|I_+|}\Bigl| y \in I_+:\, \sum_{J \in \mathcal{D}(I_+)} \alpha_J \langle f \rangle_J \mathbbm{1}_J(y) > 1 - \epsilon \Bigr| \\
		=\frac{1}{2|I_-|}\Bigl| y \in I_-:\, \sum_{J \in \mathcal{D}(I_-)} \alpha_J \langle g_- \rangle_J \mathbbm{1}_J(y) > 1 \Bigr| \\
		+\frac{1}{2|I_+|}\Bigl| y \in I_+:\, \sum_{J \in \mathcal{D}(I_+)} \alpha_J \langle g_+ \rangle_J \mathbbm{1}_J(y) > 1 \Bigr|.
	\end{multline*}
	
	Let $I_j = [e_{j}, e_{j+1})$, where $e_j = \frac{1}{2} - 2^{-j}$, and suppose that $\alpha_{\widehat{I_j}} = 0$ for $j \geq 1$ and $\alpha_{I_-} = 0$, then
	\begin{equation}\label{TechnicalLemma.Proof.eq.1}
		\frac{1}{2|I_-|}\Bigl| y \in I_-:\, \sum_{J \in \mathcal{D}(I_-)} \alpha_J \langle g_- \rangle_J \mathbbm{1}_J(y) > 1 \Bigr| = \frac{1}{2}\sum_{j= 1}^\infty 2^{-j}\frac{1}{|I_j|}\Bigl| y\in I_j:\, \mathcal{A}(g_-\mathbbm{1}_j)(y) > 1 \Bigr|.
	\end{equation}
	
	Let $\theta = 1-2\alpha$ and write
	\[
		\theta = \sum_{j=1}^\infty 2^{-j}b_j
	\]
	for some binary sequence $\{b_j\}_{j \in \mathbb{N}}$ (i.e.: write $\theta$ in binary).
	
	For a given interval $J$ let $S_Jf$ be the scaled version of $f$ adapted to $J$, i.e.: if $J = [a,b)$ then
	\[
		S_J f(x) = f\Bigl( \frac{x-a}{b-a} \Bigr).
	\]
	Abusing notation, let us also denote by $S_J \alpha$ the scaled version of the Carleson sequence $\alpha$ to the dyadic subinterval $J$ of $I$, then we have
	\begin{multline*}
		\frac{1}{|J|}\Bigl|\Bigl\{y \in J:\, \sum_{K \in \mathcal{D}(J)} (S_J \alpha)_K \langle S_J f \rangle_K \mathbbm{1}_K(y) > 1 \Bigr\}\Bigr|= \\
		\frac{1}{|I|}\Bigl| \Bigl\{ y\in I:\, \sum_{K \in \mathcal{D}(I)} \alpha_K \langle f \rangle_K \mathbbm{1}_K(y) > 1 \Bigr\} \Bigr|.
	\end{multline*}
	
	Suppose that $(1-\epsilon)f$, when restricted to $I_j$, agrees with $S_{I_j}f$ for all $j \geq 1$ such that $b_j = 1$ and is $0$ otherwise. Suppose furthermore that the Carleson sequence $\alpha$ 
	also satisfies the same similarity, i.e.: if we scale to $I$ the restriction of $\alpha$ to $I_j$
	we obtain $\alpha$ again. If we denote by $\Xi$ the left-hand side in \eqref{TechnicalLemma.eq.2} then we could use \eqref{TechnicalLemma.Proof.eq.1} to obtain
	\[
		\Xi = \frac{1}{2}\sum_{j=1}^\infty 2^{-j} b_j \Xi + \frac{1}{2|I_+|}\Bigl| y \in I_+:\, \sum_{J \in \mathcal{D}(I_+)} \alpha_J \langle g_+ \rangle_J \mathbbm{1}_J(y) > 1 \Bigr|,
	\]
	hence
	\begin{align*}
		\Xi &= \Bigl(\frac{1}{1+2\alpha}\Bigr)\frac{1}{|I_+|}\Bigl| y \in I_+:\, \sum_{J \in \mathcal{D}(I_+)} \alpha_J \langle g_+ \rangle_J \mathbbm{1}_J(y) > 1 \Bigr| \\
		&= \Bigl(\frac{x}{x+2\epsilon}\Bigr)\frac{1}{|I_+|}\Bigl| y \in I_+:\, \sum_{J \in \mathcal{D}(I_+)} \alpha_J \langle g_+ \rangle_J \mathbbm{1}_J(y) > 1 \Bigr|,
	\end{align*}
	which is what we wanted.
	Note also that we could use the same method to compute the average of $f$ and it yields precisely the right amount: $x$.
	
	Therefore we just have to show that we can find a function $f$ and a Carleson sequence $\alpha$ satisfying these self-similarity conditions. Let us start with $f$: define the operator $T$ by
	\[
		Tf = (1-\epsilon)\sum_{j=1}^\infty b_j \mathbbm{1}_{I_j} S_{I_j} f + (1-\epsilon)\mathbbm{1}_{I_+} g_+.
	\]
	We need to show that $T$ has a fixed point in $L^1(I)$; we will do this following the steps of the proof of the Banach fixed point theorem.
	Let $f_0 = (1-\epsilon)g_+ \mathbbm{1}_{I_+}$ and define inductively
	\[
		f_{n+1} = Tf_n.
	\]
	We should show that $f_n$ is a Cauchy sequence in $L^1(I)$, but observe that
	\begin{align*}
		\|f_{n+1}-f_n\|_{L^{1}(I)} &= (1-\epsilon)\int_{I_-} \Bigl| \sum_{j=1}^\infty b_j \mathbbm{1}_{I_j} S_{I_j}(f_n) - \sum_{j=1}^\infty b_j \mathbbm{1}_{I_j} S_{I_j}(f_{n-1}) \Bigr| \\
		&= (1-\epsilon)\sum_{j=1}^\infty b_j \int_{I_j} |S_{I_j}(f_n) - S_{I_j}(f_{n-1})| \\
		&= (1-\epsilon)\sum_{j=1}^\infty b_j|I_j| \int_{I}|f_n-f_{n-1}| \\
		&= (1-\epsilon)\int_I |f_n-f_{n-1}| \sum_{j=1}^\infty b_j 2^{-j-1} \\
		&= \frac{(1-\epsilon)(1-2\alpha)}{2} \int_{I} |f_n-f_{n-1}|.
	\end{align*}
	The constant $\xi := \frac{(1-\epsilon)(1-2\alpha)}{2}$ is strictly less than $1$ and by induction we have
	\[
		\|f_{n+1}-f_n\|_{L^1(I)} \lesssim \xi^{n},
	\]
	hence the sequence is Cauchy. This finishes the proof of existence for $f$ since we can just define $f$ to be the limit in $L^1$ of the sequence $f_n$ defined above.

	To show the existence of the Carleson sequence we can follow the same steps as above, but now we don't have to deal with convergence issues. Indeed, start with a sequence as in the beginning of the proof
	and define inductively the $(n+1)$-th sequence $\alpha^{n+1}$ by inserting the entire dyadic tree of $\alpha^n$ at each $I_j$. At each step we are only changing the value of the sequence at deeper and deeper levels,
	so we can just define $\alpha_K$ as the the value of $\alpha_K^n$, where $n$ is the first integer at which the sequence $\alpha_K^n$ stabilizes.
\end{proof}

We are now ready to prove Theorem \ref{MainTheorem2}, we will use the same ideas and notation as in the proof of Theorem \ref{Optimality}. Given $\epsilon$, $I$ and $x \in (0,1)$ let $N$ be the smallest integer such that
\[
	\delta^n \geq \frac{2-\epsilon}{1+x(1-\epsilon)}.
\]
Let $f_1$ be the constant function $x$ on $I_+$ and let $\alpha^1$ be the one-term Carleson sequence which is $1$ at $I_+$. Now define the function $f_{n+1}$ and Carleson sequence $\alpha^{n+1}$
inductively by applying Lemma \ref{TechnicalLemma} to the function $g_+ := S_{I_+}(f_n)$ and the Carleson sequence $S_{I_+}(\alpha^n)$; let $f = f_N$ and $\alpha = \alpha^N$. Then, as in the proof of Theorem \ref{Optimality}, we have 
\[
	\frac{1}{|I|}\Bigl| \Bigl\{ y\in I:\, \sum_{J \in \mathcal{D}(I)} \alpha_J \langle f \rangle_J \mathbbm{1}_J >1 \Bigr\} \Bigr| \geq \exp\Bigl( \sum_{j=0}^N \frac{2\epsilon}{1-\delta^j(1+x)} \Bigr),
\]
which we showed to be
\[
	\frac{2x}{x+1} + O(\epsilon),
\]
and this is what we wanted to prove.

\bibliography{../bibliography}{}
\bibliographystyle{abbrv}

\end{document}